\documentclass[12pt, a4paper]{article}
\usepackage{graphicx}

\usepackage[a4paper, top=3cm, left=3cm, right=2cm, bottom=3cm]{geometry}
\usepackage{amsmath, amsthm}
\usepackage{amssymb}
\usepackage{multicol,comment}
\usepackage{amsfonts,graphicx,color}
\usepackage{color, enumerate, fancyhdr, dsfont}
\usepackage[utf8]{inputenc}
\usepackage{mathrsfs}
\usepackage{float}
\usepackage[normalem]{ulem}
\usepackage{pb-diagram}
\usepackage[T1]{fontenc}

\usepackage{a4wide}
\usepackage{graphicx}
\usepackage{palatino}
\usepackage{multirow}
\usepackage{booktabs}
\usepackage{titlesec}
\usepackage{acronym}
\usepackage{enumitem}

\usepackage{thmtools,hyperref,cleveref}
\hypersetup{
    colorlinks=true,
    linkcolor=blue,
    filecolor=magenta,
    urlcolor=blue,
    citecolor=blue,
}

    \newcommand{\Q}{\mathbb{Q}}

    \newcommand{\R}{\mathbb{R}}

    \newcommand{\ran}{\mbox{\rm ran}}

    \newcommand{\Cwf}{\mathcal{C}}

    \newcommand{\Iwf}{\mathcal{I}}
    
    \newcommand{\Kwf}{\mathcal{K}}
    
    \newcommand{\Mwf}{\mathcal{M}}
    \newcommand{\Nwf}{\mathcal{N}}

    \newcommand{\bfrak}{\mathfrak{b}}
    \newcommand{\cfrak}{\mathfrak{c}}
    \newcommand{\dfrak}{\mathfrak{d}}

    \newcommand{\rfrak}{\mathfrak{r}}
    \newcommand{\sfrak}{\mathfrak{s}}

    \newcommand{\xfrak}{\mathfrak{x}}
    \newcommand{\yfrak}{\mathfrak{y}}

    \newcommand{\add}{\mbox{\rm add}}
    \newcommand{\cov}{\mbox{\rm cov}}
    \newcommand{\non}{\mbox{\rm non}}
    \newcommand{\cof}{\mbox{\rm cof}}

    \newcommand{\Ed}{\mathbf{Ed}}
    \newcommand{\Mbf}{\mathbf{M}}
    \newcommand{\Vbf}{\mathbf{V}}
    \newcommand{\Nbf}{\mathbf{N}}

    \newcommand{\Cn}{\mathbf{Cn}}
    
    \newcommand{\Bor}{\mathbb{B}}
    \newcommand{\Cor}{\mathbb{C}}

    \newcommand{\Por}{\mathbb{P}}

    \newcommand{\Ior}{\mathbb{I}}

    \newcommand{\Sbf}{\mathbf{S}}
    \newcommand{\Cf}{\mathbf{Cf}}

\title{Tukey-order with models on Palikowski's theorems}

\author{Miguel A. Cardona}
\date{Institute of Discrete Mathematics and Geometry.\\
Faculty of Mathematics and Geoinformation.\\
\addvspace{\medskipamount}
TU Wien.}

\begin{document}

\newcounter{enuAlph}
\renewcommand{\theenuAlph}{\Alph{enuAlph}}

\makeatletter
\def\@roman#1{\romannumeral #1}
\makeatother

\theoremstyle{plain}
  \newtheorem{theorem}{Theorem}[section]
  \newtheorem{corollary}[theorem]{Corollary}
  \newtheorem{lemma}[theorem]{Lemma}
  \newtheorem{mainlemma}[theorem]{Main Lemma}
  \newtheorem{maintheorem}[enuAlph]{Main Theorem}
  \newtheorem{prop}[theorem]{Proposition}
  \newtheorem{claim}[theorem]{Claim}
  \newtheorem{exer}[theorem]{Exercise}
  \newtheorem{question}[theorem]{Question}
   \newtheorem{fact}[theorem]{Fact}
  \newtheorem{problem}[theorem]{Problem}
  \newtheorem{conjecture}[theorem]{Conjecture}
  \newtheorem*{thm}{Theorem}
  \newtheorem{teorema}[enuAlph]{Theorem}
  \newtheorem*{defn*}{Definition}
  \newtheorem*{corolario}{Corollary}
\theoremstyle{definition}
  \newtheorem{definition}[theorem]{Definition}
  \newtheorem{example}[theorem]{Example}
  \newtheorem{remark}[theorem]{Remark}
  \newtheorem{notation}[theorem]{Notation}
  \newtheorem{context}[theorem]{Context}

\newcommand{\azul}[1]{{\color{blue}#1}}
\newcommand{\rojo}[1]{{\color{red}#1}}
\newcommand{\verdeone}[1]{{\color{green}#1}}
\newcommand{\tachar}[1]{{\color{red}\sout{#1}}}
\definecolor{amber}{rgb}{1.0,0.49,0.0}
\definecolor{huntergreen}{rgb}{0.21, 0.37, 0.23}
\definecolor{ogreen}{RGB}{107,142,35}
\definecolor{bostonuniversityred}{rgb}{0.8, 0.0, 0.0}

\newcommand{\ogreen}[1]{{\color{ogreen}#1}}
\newcommand{\amber}[1]{{\color{amber}#1}}
\newcommand{\huntergreen}[1]{{\color{huntergreen}#1}}
\newcommand{\bostonuniversityred}[1]{{\color{bostonuniversityred}#1}}

\newcommand{\Fn}{\mathrm{Fn}}
\newcommand{\leqT}{\preceq_{\mathrm{T}}}
\newcommand{\eqT}{\cong_{\mathrm{T}}}
\newcommand{\la}{\langle}
\newcommand{\ra}{\rangle}
\newcommand{\id}{\mathrm{id}}
\newcommand{\Lv}{\mathrm{Lv}}
\newcommand{\sig}{\boldsymbol{\Sigma}}
\newcommand{\spl}{\mathrm{spl}}
\newcommand{\st}{\mathrm{st}}
\newcommand{\suc}{\mathrm{succ}}
\newcommand{\cosig}{\boldsymbol{\Pi}}
\newcommand{\Lb}{\mathrm{Lb}}
\newcommand{\pw}{\mathrm{pw}}
\newcommand{\Lc}{\mathbf{Lc}}

\newcommand{\blc}{\mathfrak{b}^{\mathrm{Lc}}}
\newcommand{\dlc}{\mathfrak{d}^{\mathrm{Lc}}}
\newcommand{\Cv}{\mathbf{Cv}^\exists}
\newcommand{\aLc}{\mathbf{aLc}}
\newcommand{\balc}{\mathfrak{b}^{\mathrm{aLc}}}
\newcommand{\dalc}{\mathfrak{d}^{\mathrm{aLc}}}

\newcommand{\Hcal}{\mathcal{H}}
\newcommand{\SNcal}{\mathcal{SN}}
\newcommand{\Fr}{\mathrm{Fr}}
\newcommand{\Dbf}{\mathbf{D}}
\newcommand{\Cbf}{\mathbf{C}}
\newcommand{\Rbf}{\mathbf{R}}
\newcommand{\Ibb}{\mathbb{I}}
\newcommand{\PTbb}{\mathbb{PT}}
\newcommand{\Qbb}{\mathbb{Q}}
\newcommand{\Tbb}{\mathbb{T}}
\newcommand{\Scal}{\mathcal{S}}
\newcommand{\Ctebf}{\mathbf{Cte}}
\newcommand{\sigmaf}{\sigma^f}
\newcommand{\Af}{A^f}

\maketitle

\begin{abstract}
In \cite{paw86} Pawlikowski proved that, if $r$ is a random real over $\Nbf$, and $c$ is Cohen real over $\Nbf[r]$, then 
\begin{itemize}
    \item[(a)] in $\Nbf[r][c]$ there is a Cohen real over $\Nbf[c]$, and 
    \item[(b)] $2^\omega\cap\Nbf[c]\notin\Nwf\cap\Nbf[r][c]$, so in $\Nbf[r][c]$ there is no random real over $\Nbf[c]$. 
\end{itemize}
To prove this, Pawlikowski proposes the following notion: Given two models $\Nbf\subseteq \Mbf$ of ZFC,  we associate with a cardinal characteristic $\xfrak$ of the continuum, a sentence $\xfrak_\Nbf^\Mbf$ saying that, in $\Mbf$, the reals in $\Nbf$ give an example of a family fulfilling the requirements of the cardinal. So to prove (a) and (b), it suffices to prove that 
\begin{itemize}
    \item[(a')] $\cov(\Mwf)_{\Nbf[c]}^{\Mbf[c]}\Rightarrow\cof(\Mwf)_{\Nbf}^{\Mbf}\Rightarrow\cov(\Nwf)_{\Nbf}^{\Mbf}$, and 
    \item[(b')] $\cov(\Mwf)_\Nbf^\Mbf\Rightarrow\add(\Mwf)_{\Nbf}^{\Mbf}\Rightarrow\non(\Mwf)_{\Nbf[c]}^{\Mbf[c]}\Rightarrow\cov(\Nwf)_{\Nbf[c]}^{\Mbf[c]}$.
\end{itemize}
In this paper we introduce the notion of Tukey-order with models, which expands the concept of Tukey-order introduced by Vojt\'{a}\v{s} \cite{V}, to prove expressions of the form $\xfrak_\Nbf^\Mbf\Rightarrow\yfrak_\Nbf^\Mbf$. In particular, we show (a') and (b') using Tukey-order with models.

\end{abstract}

\section{Introduction}\label{SecIntro}

  
Let $\Nwf$ be the $\sigma$-ideal of measure zero subsets of $2^\omega$, $\Mwf$ the $\sigma$-ideal of meager sets in $2^\omega$, let $\Kwf$ be the $\sigma$-ideal generated by the subsets of $\R$ whose intersection with $\Q^*$ (the set of irrational numbers) is compact in $\Q^*$, and let  $\Cwf$ be the $\sigma$-ideal of countable subsets of reals. It is well-known that $\add(\Kwf)=\non(\Kwf)=\bfrak$, $\add(\Cwf)=\non(\Cwf)=\aleph_1$, $\cov(\Kwf)=\cof(\Kwf)=\dfrak$, and $\cov(\Cwf)=\cof(\Cwf)=\cfrak$, where $\bfrak$, $\dfrak$ and $\cfrak$ are the bounding and dominating numbers, and the size of $\R$, respectively. These cardinals describe the entries in Cicho\'n's diagram.

\begin{figure}[ht]
  \begin{center}
   \includegraphics[scale=1.0]{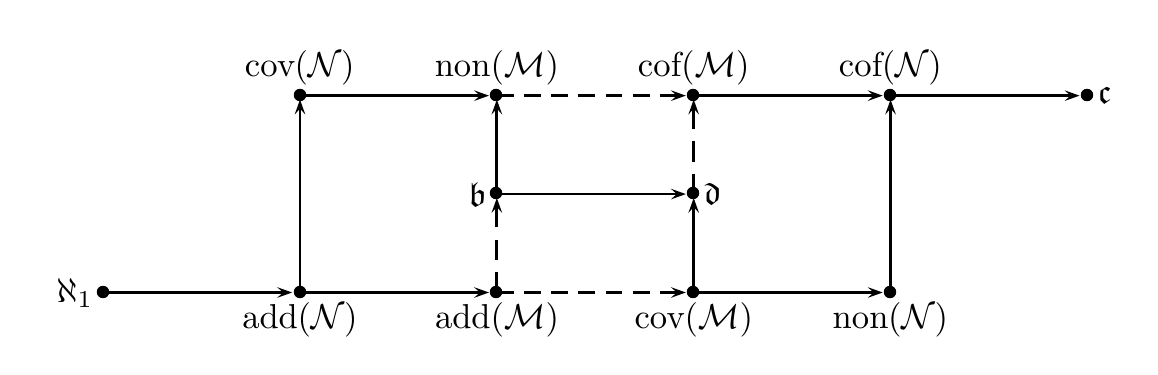}
   \caption{Cicho\'n's diagram.
   An arrow indicates that ZFC proves $\leq$ between both cardinals. In addition,
    in the diagram, the dashed arrows mean
   $\add(\Mwf)=\min\{\bfrak,\cov(\Mwf)\}$ and $\cof(\Mwf)=\max\{\dfrak,\non(\Nwf)\}$.}
     \label{Figcichon}
  \end{center}
\end{figure}

A \textit{relational system} is a triple $\Rbf=\la X, Y, \sqsubset\ra$ where $\sqsubset$ is a relation contained in $X\times Y$. Such a relational system has two cardinal invariants associated with it:

\begin{align*}
 \bfrak(\Rbf)&:=\min\{|F|:F\subseteq X\text{\ and }\neg\exists y\in Y\forall x\in F(x \sqsubset y)\},\\
 \dfrak(\Rbf)&:=\min\{|D|:D\subseteq Y\text{\ and }\forall x\in X \exists y\in D(x \sqsubset y)\}.
\end{align*}

Denote $\Rbf^{\perp}:=\la Y, X,\sqsubset^{\perp}\ra$ where $y \sqsubset^{\perp}x$ iff $\neg(x \sqsubset y)$.  

We say that a relational system $\Rbf=\la X, Y, \sqsubset\ra$ is $\textit{real}$-\textit{definable} if both $X$ and $Y$ are non-empty and analytic in Polish spaces  $Z$ and $W$, respectively, and  $\sqsubset$ is analytic in $Z\times W$.
\begin{definition}[Tukey-order {\cite{V}}]\label{def:tukey}
Let $\Rbf':=\la X',Y', \sqsubset'\ra$ be another relational system. If $\Psi_-$, $\Psi_+$ are a pair of mappings $\Psi_-:X\to X'$, and $\Psi_+:Y'\to Y$ such that, for any $x\in X$ and $y'\in Y'$, $\Psi_-(x) \sqsubset'y'$ implies $x \sqsubset \Psi_+(y')$, then
we say that \emph{$\Rbf$ is Tukey-below $\Rbf'$}, denoted by $\Rbf\leqT\Rbf'$. Say that \emph{$\Rbf$ and $\Rbf'$ are Tukey-equivalent}, denoted by $\Rbf\eqT\Rbf'$, if $\Rbf\leqT\Rbf'$ and $\Rbf'\leqT\Rbf$. Note that $\Rbf\leqT\Rbf'$ implies $\bfrak(\Rbf')\leq\bfrak(\Rbf)$ and $\dfrak(\Rbf)\leq\dfrak(\Rbf')$.
\end{definition}

Let $\Mbf$ be a transitive model of ZFC (or of a large enough finite fragment of it). For a real $x$ denote by $B_x$ the Borel set with code $x$ and set $B_x^\Mbf=B_x\cap\Mbf$ its relativization in $\Mbf$. For a real definable system $\Rbf=\la X,Y,\sqsubset\ra$, when dealing with $\Rbf$ inside some model $\Mbf$, we look at its interpretations $\Rbf^\Mbf=\la X^\Mbf,Y^\Mbf,\sqsubset^\Mbf\ra$. If $\Iwf$ is an ideal let $\Iwf^\Mbf$ be the family of members of $\Iwf$ whose members can be covered by a Borel set in $\Iwf$ coded in $\Mbf$. If $\Nbf\subseteq \Mbf$ are two models of ZFC we associate with each cardinal characteristic in Cicho\'n's diagram a sentence saying that $\Nbf$ gives an example in $\Mbf$ of a family fulfilling the requirements of the cardinal, that is, for $\Iwf\in\{\Nwf,\Mwf\}$ define the properties:
\begin{align*}
\dfrak_\Nbf^\Mbf &:\textrm{any function from $\omega^\omega\cap\Mbf$ is dominated by a function from $\omega^\omega\cap\Nbf$}.\\
\bfrak_\Nbf^\Mbf &:\textrm{there is no function from $\omega^\omega\cap\Mbf$ that dominates all functions from $\omega^\omega\cap\Nbf$}.\\
\add(\Iwf)_\Nbf^\Mbf&: \bigcup \Iwf^\Nbf\not\in\Iwf^\Mbf.\\
\cov(\Iwf)_\Nbf^\Mbf&:\, \bigcup \Iwf^\Nbf\supseteq2^\omega\cap\Mbf.\\
\non(\Iwf)_\Nbf^\Mbf&:\, 2^{\omega}\cap\Nbf\notin\Iwf^\Mbf.\\
\cof(\Iwf)_\Nbf^\Mbf&:\, \Iwf^\Nbf \textrm{\ is cofinal in\ } \Iwf^\Mbf.
\end{align*}
Cicho\'n's and Pawlikowski \cite{cicpal} investigated  the effect  on the cardinal characteristics in Cicho\'n's diagram after adding a single Cohen real or one random real. Motivated by this investigation, Palikowski \cite{paw86} formulated and proved that, if $c$  and $r$ are a Cohen real over $\Mbf$ and  a random real over $\Mbf$ respectively, then 
    \begin{itemize} 
    \item[(C1)] $\cof(\Mwf)_{\Nbf}^{\Mbf}\Leftrightarrow\cof(\Mwf)_{\Nbf[c]}^{\Mbf[c]}\Leftrightarrow\dfrak_{\Nbf[c]}^{\Mbf[c]}\Leftrightarrow\cov(\Mwf)_{\Nbf[c]}^{\Mbf[c]}$.
    \item[(C2)] $\add(\Mwf)_{\Nbf}^{\Mbf}\Leftrightarrow\add(\Mwf)_{\Nbf[c]}^{\Mbf[c]}\Leftrightarrow\bfrak_{\Nbf[c]}^{\Mbf[c]}\Leftrightarrow\non(\Mwf)_{\Nbf[c]}^{\Mbf[c]}$.
   \item[(C3)] $\add(\Nwf)_{\Nbf}^{\Mbf}\Rightarrow\add(\Nwf)_{\Nbf[c]}^{\Mbf[c]}\Leftrightarrow\cov(\Nwf)_{\Nbf[c]}^{\Mbf[c]}$.
   \item[(C4)] $\non(\Nwf)_{\Nbf[c]}^{\Mbf[c]}\Leftrightarrow\cof(\Nwf)_{\Nbf[c]}^{\Mbf[c]}\Rightarrow\cof(\Nwf)_{\Nbf}^{\Mbf}$.
\end{itemize}
\begin{itemize} 
    \item[(R1)] $\add(\Nwf)_{\Nbf}^{\Mbf}\Leftrightarrow\add(\Nwf)_{\Nbf[r]}^{\Mbf[r]}$ and $\cof(\Nwf)_{\Nbf}^{\Mbf}\Leftrightarrow\cof(\Nwf)_{\Nbf[r]}^{\Mbf[r]}$.
    \item[(R2)] $\bfrak_{\Nbf}^{\Mbf}\Leftrightarrow\bfrak_{\Nbf[r]}^{\Mbf[r]}$ and $\dfrak_{\Nbf}^{\Mbf}\Leftrightarrow\dfrak_{\Nbf[r]}^{\Mbf[r]}$.
   \item[(R3)] $\dfrak_{\Nbf}^{\Mbf}\Rightarrow\non(\Nwf)_{\Nbf[r]}^{\Mbf[r]}$ and $\cov(\Nwf)_{\Nbf[r]}^{\Mbf[r]}\Rightarrow\bfrak_{\Nbf}^{\Mbf}$. 
    \item[(R4)] $\cov(\Mwf)_{\Nbf[r]}^{\Mbf[r]}\Rightarrow\cov(\Mwf)_{\Nbf}^{\Mbf}$ and $\non(\Mwf)_{\Nbf}^{\Mbf}\Rightarrow\non(\Mwf)_{\Nbf[r]}^{\Mbf[r]}$.
    \item[(R5)] $\cov(\Nwf)_{\Nbf[r]}^{\Mbf[r]}\Rightarrow\cov(\Nwf)_{\Nbf}^{\Mbf}$ and $\non(\Nwf)_{\Nbf}^{\Mbf}\Rightarrow\non(\Nwf)_{\Nbf[r]}^{\Mbf[r]}$.
    \item[(R6)]  $\add(\Mwf)_{\Nbf[r]}^{\Mbf[r]}\Leftrightarrow\add(\Mwf)_{\Nbf}^{\Mbf}$ and $\cof(\Mwf)_{\Nbf}^{\Mbf}\Leftrightarrow\cof(\Mwf)_{\Mbf[r]}^{\Mbf[r]}$.
\end{itemize}
Later, Bartoszy\'nki, Roslanowski, and Shelah \cite{brs} proved the converse of (R4):
\begin{itemize} 
\item[(R7)] $\cov(\Mwf)_{\Nbf}^{\Mbf}\Rightarrow\cov(\Mwf)_{\Nbf[r]}^{\Mbf[r]}$ and $\non(\Mwf)_{\Nbf[r]}^{\Mbf[r]}\Rightarrow\non(\Mwf)_{\Nbf}^{\Mbf}$.
\end{itemize}
This research was completed by Shelah \cite[Lemma 1.3.4]{BJ} who proved that  
\begin{itemize}
    \item[(R8)] $\rfrak_{\Nbf[r]}^{\Mbf[r]}\Rightarrow\dfrak_\Nbf^\Mbf$ and $\bfrak_\Nbf^\Mbf\Rightarrow\sfrak_{\Nbf[r]}^{\Mbf[r]}$ (where $\rfrak$ and $\sfrak$ are \textit{the splitting number} and \textit{unreaping number}, respectively, see \autoref{examplebasic}(v)).
\end{itemize}
We summarize these implications in \autoref{cichoncohen} and \autoref{cichonrandom}.

The proof of the implication $\dfrak_{\Nbf[c]}^{\Mbf[c]}\Rightarrow\dfrak_{\Nbf}^{\Mbf}$ (in \cite[Lemma 3.4]{cicpal}) is reviewed as follows: For each $f\in\Nbf[c]\cap\omega^\omega$ find  $g_f\in\Nbf\cap\omega^\omega$ in such way that each $h\in\Mbf\cap\omega^\omega$ dominated by $f$ is also dominated by $g_f$. In other words, if $h\in\omega^\omega\cap\Mbf$ is not dominated by any function from $\omega^\omega\cap\Nbf$ then it is also not dominated by any function from $\Nbf[c]\cap\omega^\omega$. 

A curious aspect of the argument above, established by Cicho\'n and Palikowski, gives additional information beyond of the implications. They get (implicitly) two maps 
$\Psi_-:\omega^\omega\cap \Mbf\to\omega^\omega\cap \Mbf[c]$ and $\Psi_+:\omega^\omega\cap \Nbf[c]\to \Nbf\cap\omega^\omega$ such that, for any $h\in\omega^\omega\cap\Mbf$ and $f\in \Nbf[c]\cap\omega^\omega$, $\Psi_-(h)=h$ and if $h$ is dominated by $f$ then it is also dominated by $\Psi_+(f)$, which resembles the Tukey-order. This can be rephrased in the language of Tukey-order (see \autoref{def:tukey}), that is,  $\la\omega^\omega\cap\Mbf,\omega^\omega\cap\Nbf,\leq^*\ra\leqT\la\omega^\omega\cap\Mbf[c],\omega^\omega\cap\Nbf[c],\leq^*\ra$. This rephrasing is important because we obtain in this way a simple description to treat implications between  sentences involving cardinal characteristics.

Motivated by the above description we expland the concept of Tukey-order with models as follows:

Assume that $\Nbf\subseteq \Mbf$ are models of ZFC. For a  real definable relational system $\Rbf$ we let
\begin{itemize}
\item[(i)] $\dfrak(\Rbf)_{\Nbf}^{\Mbf}$ iff $\forall x\in X^\Mbf\exists y\in Y^\Nbf(x \sqsubset y)$. 
\item[(ii)] $\bfrak(\Rbf)_{\Nbf}^{\Mbf}$ iff $\neg\exists y\in Y^\Mbf\forall x\in X^\Nbf(x \sqsubset y)$.
\end{itemize}

\begin{definition}\label{modeltukey}
Let $\Nbf_0\subseteq \Mbf_0$, $\Nbf\subseteq \Mbf$ be models of ZFC and let $\Rbf, \Rbf'$ be two real definable systems. We write $\Rbf\preceq_{\Nbf_0,\Nbf}^{\Mbf_0,\Mbf}\Rbf'$ if there is a pair of maps $\Psi_-:X^{\Mbf_0}\to X'^\Mbf$ and $\Psi_+:Y'^\Mbf\to Y^{\Mbf_0}$ such that 
\begin{itemize}
    \item[(a)] for all $x\in X^{\Mbf_0}$ and for all $y'\in Y'^\Mbf$,
    $\Psi_-(x)\sqsubset'y'$ implies $x \sqsubset\Psi_+(y')$.
    \item[(b)] $\Psi_-[X^{\Nbf_0}]\subseteq X'^\Nbf$ and $\Psi_+[Y'^{\Nbf}]\subseteq Y^{\Nbf_0}$.
\end{itemize}
Also we write $\Rbf\cong_{\Nbf_0,\Nbf}^{\Mbf_0,\Mbf}\Rbf'$ if $\Rbf\preceq_{\Nbf_0,\Nbf}^{\Mbf_0,\Mbf}\Rbf'$ and $\Rbf'\preceq_{\Nbf,\Nbf_0}^{\Mbf,\Mbf_0}\Rbf$.
\end{definition}

So this definition formalices the above example. The reason for considering the Tukey-order with models is the following lemma: 

\begin{lemma}\label{lem:tukey}
Assume that $\Nbf_0\subseteq \Mbf_0$, $\Nbf\subseteq \Mbf$ are models of ZFC and let $\Rbf, \Rbf'$ be two real definable relational systems. If $\Rbf\preceq_{\Nbf_0,\Nbf}^{\Mbf_0,\Mbf}\Rbf'$, then   $\dfrak(\Rbf')_{\Nbf}^{\Mbf}\Rightarrow\dfrak(\Rbf)_{\Nbf_0}^{\Mbf_{0}}$ and $\bfrak(\Rbf)_{\Nbf_{0}}^{\Mbf_{0}}\Rightarrow\bfrak(\Rbf')_{\Nbf}^{\Mbf}$. 
\end{lemma}
\begin{proof}
According to \autoref{modeltukey} choose functions $\Psi_-:X^{\Mbf_0}\to X'^{\Mbf}$ and $\Psi_+:Y^{\Mbf}\to Y^{\Mbf_0}$ fulfilling (a)-(b). We only prove $\dfrak(\Rbf')_{\Nbf}^{\Mbf}\Rightarrow\dfrak(\Rbf)_{\Nbf_0}^{\Mbf_{0}}$, since the second statement is analogous. To this end assume that $\dfrak(\Rbf')_{\Nbf}^{\Mbf}$ holds and show $\dfrak(\Rbf)_{\Nbf_0}^{\Mbf_{0}}$. 

Let $x\in X^{\Mbf_0}$ be arbitrary. There is a $y\in Y'^\Nbf$ such that $\Psi_-(x) \sqsubset'y$. Now, by \autoref{modeltukey}(a) we get $x\sqsubset\Psi_+(y)$ and $\Psi_+\in Y^{\Nbf_0}$ by (b).
\qedhere{\textrm{\ (\autoref{lem:tukey}})}
\end{proof}
\textbf{Objective.} The main motivation of this work is to prove some of the implications from \autoref{cichoncohen} and \autoref{cichonrandom} using Tukey-order with models.



\begin{figure}[!ht]
  \begin{center}
    \includegraphics[scale=0.75]{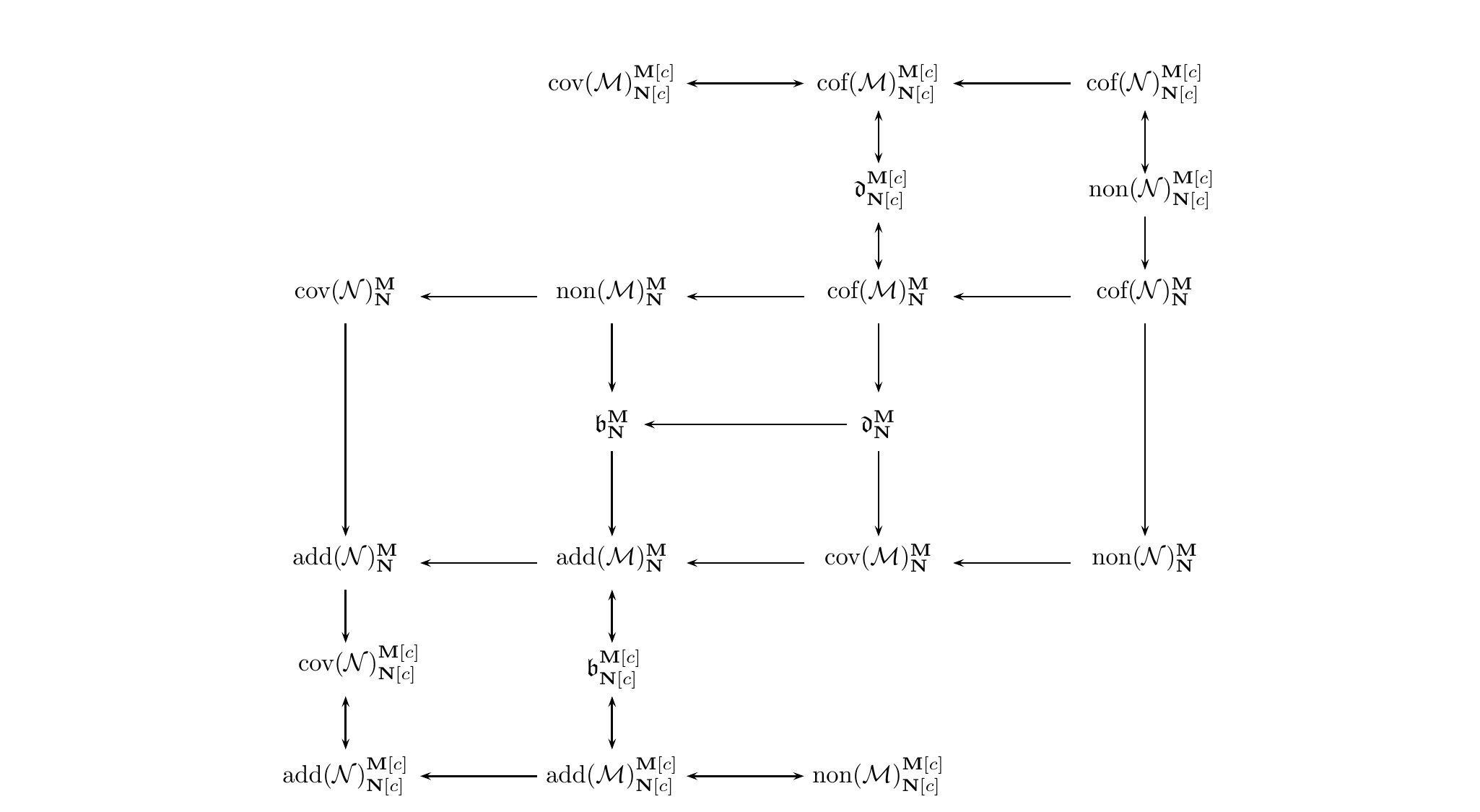}
  \caption{Cicho\'n's diagram with models after adding a Cohen real.}
  \label{cichoncohen}
  \end{center}
\end{figure}

\begin{figure}[!ht]
  \begin{center}
    \includegraphics[scale=0.75]{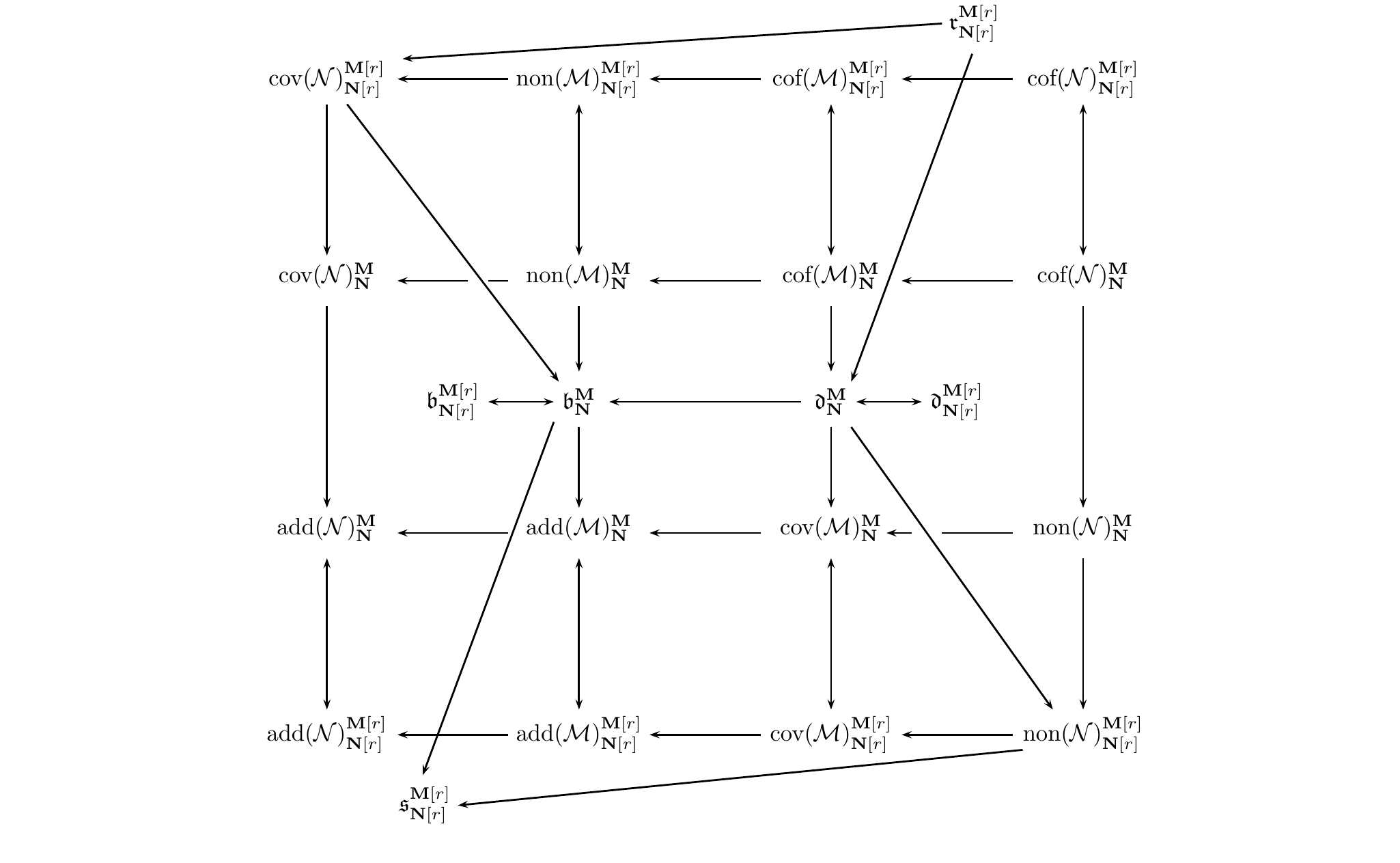}
  \caption{Cicho\'n's diagram with models after adding a random real.}
  \label{cichonrandom}
  \end{center}
\end{figure}

This paper is structured as follows: We  review  in \autoref{SecPrel} the  basic  notation and the results this paper is based on. We deal with the concept of  $\dfrak(\Rbf)_{\Nbf}^{\Mbf}$ and $\bfrak(\Rbf)_{\Nbf}^{\Mbf}$ in \autoref{bd}. We show in  \autoref{dbcohenrandom} the effect on $\dfrak(\Rbf)_{\Nbf}^{\Mbf}$ and $\bfrak(\Rbf)_{\Nbf}^{\Mbf}$ after of adding  a single Cohen real without goodness. Likewise after a single random real.

\section{Preliminaries}\label{SecPrel}
For a set $A\subseteq2^\omega\times2^\omega$ denote $A_x=\{y:\la x,y\ra\in A\}$ and $A^y=\{x:\la x,y\ra\in A\}$. Denote by $X^\omega$ the set of all maps from $\omega$ into $X$ considered as sequences of elements of $X$. 

Given a function $b$ with domain $\omega$ such that $b(i)\neq\emptyset$ for all $i<\omega$, $h\in \omega^\omega$ and $n<\omega$, define $\Scal(b,h)=\prod_{n<\omega}[b(n)]^{\leq h(n)}$. A \emph{slalom} is a function $\varphi:\omega\to[\omega]^{<\omega}$.

For functions $f$, $g\in\omega^\omega$ and $\varphi$ with domain $\omega$ we define 
\begin{itemize} 
    \item[(1)] $f\leq^*g$ iff $\exists n\in\omega\forall m\geq n(f(m)\leq g(m))$.
    \item[(2)] $f\neq^* g$ iff $\exists n\in\omega \forall m\geq n(f(m)\neq g(m))$.
    \item[(3)]  For a slalom $\varphi$, define 
    \begin{enumerate}
        \item[(i)] $f\in^{*}\varphi$ by $\exists m\in\omega\forall n\geq  m(f(n)\in \varphi(n))$, which is read \textit{$\varphi$ localizes $x$};
        \item[(ii)] $f \in^{\infty}\varphi$ iff $\forall n\in\omega\exists m\geq n(f(n)\in \varphi(n))$. Denote its negation by $f\not\in^\infty \varphi$, which is read \emph{$\varphi$ anti-localizes $f$}.
    \end{enumerate}
\end{itemize}

For $A, B\in[\omega]^{\aleph_0}$, define $A\propto B$ iff either $B\smallsetminus A$ is finite or $A\cap B$ is finite. Note that $A\not\propto B$ iff $A$ \textit{splits} $B$, that is $A\cap B$ and $B\smallsetminus A$ are infinite. Denote by $\Ior$ the set of interval partitions of $\omega$. For $I, J\in\Ior$, define  
\begin{itemize}
    \item[(i)] $I\sqsubseteq^{\Ior} J$ iff $\forall^{\infty}n\exists m(I_m\subseteq J_n)$.
    \item[(ii)] $I \ntriangleright J \text{\ iff } \forall^\infty n\, \forall m\, ( I_n\nsupseteq J_m)$
\end{itemize}
Let $\Por$ be a poset. For a model $\Mbf$ and a set $X$ denote $\Mbf^{\Por}:=\{\tau\in\Mbf:\textrm{$\tau$ is a $\Por$-name}\}$ and $\Mbf^{\Por}_{X}:=\{ \tau\in\Mbf^{\Por}:\,\,\Vdash \tau\in X\}$. Say that $\Por$ is a \textit{Suslin ccc forcing notion} if it is ccc and there is a Polish space $Z$ such that
\begin{itemize}
\item[(i)] $\Por\subseteq Z$,
\item[(ii)] $\leq_{\Por}\,\subseteq  Z\times Z$
is $\Sigma_1^1$ and
\item[(iii)]  $\bot_{\Por}\subseteq Z\times Z$ is $\Sigma_1^1$.
\end{itemize}
\begin{definition}
We say that $\Rbf=\la X,Y,\sqsubset\ra$ is a \textit{Polish relational system} if
\begin{enumerate}
\item[(I)] $X$ and $Y$ are Perfect Polish spaces, and
\item[(II)] $\sqsubset=\bigcup_{n\in\omega}\sqsubset_{n}$ where $\langle\sqsubset_{n}: n\in\omega\rangle$  is some increasing sequence of closed subsets of $X\times Y$ such that, for any $n<\omega$ and for any $y\in Y$,
$(\sqsubset_{n})^{y}=\{x\in X:x\sqsubset_{n}y \}$ is closed nowhere dense.
\end{enumerate}
The closed sets mentionated in (I) and (II) have an absolute definition, that is, as Borel sets they have the same Borel codes in all transitive models. Say that $x$ is \textit{$\Rbf$-unbounded over  $H$} if $\forall y \in H\cap Y(x\not\sqsubset y)$.
\end{definition}

Let $\Cor$ and $\Bor$ be the Cohen algebra and random algebra for adding one Cohen real and one random real, respectively.

Many cardinals characteristics can be described through simple relational systems. In the following example, we recall that some of the entries in Cicho\'n's diagram can be defined through simple relational systems.

\begin{example}
For any ideal $\Iwf$ on $2^{\omega}$.
\begin{itemize}
    \item[(a)] $\Cbf_\Iwf=\la2^\omega,\Iwf,\in\ra$, so $\bfrak(\Cbf_\Iwf)=\non(\Iwf)$ and $\dfrak(\Cbf_\Iwf)=\cov(\Iwf)$.
    \item[(b)] $\Iwf:=\la\Iwf,\subseteq\ra=\la\Iwf,\Iwf,\subseteq\ra$ is directed, $\bfrak(\Iwf)=\add(\Iwf)$ and $\dfrak(\Iwf)=\cof(\Iwf)$.
\end{itemize}
\end{example}

Given two relational systems $\Rbf$ and $\Rbf'$, we let $(\Rbf;\Rbf'):=\la X\times (X')^{Y},Y\times Y',\sqsubset_{;}\ra$ where  $(x,f)\sqsubset_{;}(a,b)$ means $x\sqsubset a$ and $f(a)\sqsubset^{\prime} b$. Hence $\dfrak(\Rbf;\Rbf'):=\dfrak(\Rbf)\cdot\dfrak(\Rbf')$  and $\bfrak(\Rbf;\Rbf')=\min\{\bfrak(\Rbf),\bfrak(\Rbf)\}$ by \cite[Thm. 4.11]{blass}.

We present some examples of the classical framework, that is, with instances of Polish relational systems.

\begin{example}\label{examplebasic}
The examples (i)-(v) are Polish relational systems. 
\begin{itemize} 
    \item[(i)] Combinatorial characterizations of $\dfrak$ and $\bfrak$. 
    \begin{itemize}
        \item[(a)]  Consider the relational system $\Dbf:=\la\omega^\omega,\omega^\omega,\leq^{*}\ra$. Define $\bfrak:=\bfrak(\Dbf)$ and $\dfrak:=\dfrak(\Dbf)$.
        \item[(b)] Define the relational systems $\Dbf_1:=\la\Ior,\Ior,\sqsubseteq^{\Ior}\ra$ and $\Dbf_2:=\la\Ibb,\Ibb,\ntriangleright\ra$. It was proved in \cite{blass} that $\Dbf\eqT\Dbf_1\eqT\Dbf_2$, so $\dfrak(\Dbf_1)=\dfrak(\Dbf_2)=\dfrak$ and $\bfrak(\Dbf_1)=\bfrak(\Dbf_2)=\bfrak$.
    \end{itemize}
    \item[(ii)] Combinatorial characterizations of $\cov(\Mwf)$ and $\non(\Mwf)$.
    
    \begin{itemize}
       
        \item[(a)]  Define $\Ed:=\la\omega^\omega,\omega^\omega\neq^*\ra$. Since $\Cbf_\Mwf\leqT\Ed$, $\bfrak(\Ed)\leq\non(\Mwf)$ and $\dfrak(\Ed)\leq\cov(\Mwf)$. Even more, $\bfrak(\Ed)=\non(\Mwf)$ and $\dfrak(\Ed)=\cov(\Mwf)$ (\cite[Thm. 2.4.1 \& Thm. 2.4.7]{BJ}).
        \item[(b)] Let $b$ be a function with domain $\omega$ such that $b(i)\neq\emptyset$ for all $i<\omega$, and let $h\in\omega^\omega$. Define $\aLc(b,h):=\la\Scal(b,h),\prod b,\not\ni^\infty\ra$, so put $\balc_{b,h}:=\bfrak(\aLc(b,h))$ and $\dalc_{b,h}:=\dfrak(\aLc(b,h))$, which we refer to as \emph{anti-localization cardinals}. If $h\geq^*1$ then $\aLc(\omega,h)\eqT\Ed$, so $\balc_{\omega,h}=\non(\Mwf)$ and $\dalc_{\omega,h}=\cov(\Mwf)$ (\cite[Thm. 2.4.1 \& Thm. 2.4.7]{BJ}). Here $\omega$ denotes the constant function $b(n)=\omega$. 
     \end{itemize}
     
     \item[(iii)]  Define $\Omega_n:=\{a\in [2^{<\omega}]^{<\aleph_0}:\lambda^*(\bigcup_{s\in a}[s])\leq 2^{-n}\}$ (endowed with the discrete topology) and put $\Omega:=\prod_{n<\omega}\Omega_n$ with the product topology, which is a perfect Polish space. For every $x\in \Omega$ denote $N_{x}^{*}:=\bigcap_{n<\omega}\bigcup_{s\in x(n)}[s]$, which is clearly a Borel null set in $2^{\omega}$.

     Define the relational system $\Cn:=\la \Omega, 2^\omega, \sqsubset\ra$ where $x\sqsubset z$ iff $z\notin N_{x}^{*}$. Recall that any null set in $2^\omega$ is a subset of $N_{x}^{*}$ for some $x\in \Omega$, even more $\Cn\eqT\Cbf_\Nwf^\perp$. Hence, $\bfrak(\Cn)=\cov(\Nwf)$ and $\dfrak(\Cn)=\non(\Nwf)$.

     \item[(iv)]  For each $k<\omega$ let $\id^k:\omega\to\omega$ such that $\id^k(i)=i^k$ for all $i<\omega$ and $\Hcal:=\{\id^{k+1}:k<\omega\}$. Let $\Lc^*:=\la\omega^\omega, \Scal(\omega, \Hcal), \in^*\ra$ be the Polish relational system where \[\Scal(\omega, \Hcal):=\{\varphi:\omega\to[\omega]^{<\aleph_0}:\exists{h\in\Hcal}\forall{i<\omega}(|\varphi(i)|\leq h(i))\}.\]
     As consequence of~\cite[Thm.~2.3.9]{BJ}, $\bfrak(\Lc^*)=\add(\Nwf)$ and $\dfrak(\Lc^*)=\cof(\Nwf)$. In fact, $\Nwf\eqT\Lc^*$.

     \item[(v)] Consider the relational system $\Sbf:=\la[\omega]^{\aleph_0},[\omega]^{\aleph_0},\propto\ra$, so $\bfrak(\Sbf)=\sfrak$ and $\dfrak(\Sbf)=\rfrak$, which are known as \textit{the splitting number} and \textit{unreaping number}, respectively.
     
     \item[(vi)] Denote $\Xi:=\{f:2^{<\omega}\to2^{<\omega}:\forall s\in2^{<\omega}(s\subseteq f(s))\}$.  For $f\in\Xi$ define $G_f:=\bigcap_{n<\omega}\bigcup_{|s|\geq n}[f(s)]$. Define the relational system
     $\Cf:=\la \Xi,\Xi,\sqsubset^\Mwf\ra$ where$f \sqsubset^\Mwf g$ iff $G_g\subseteq G_f$. Recall that any element of $\Xi$ codes a member of $\Mwf$, even more $\Cf\eqT\Mwf$. Note that $\sqsubset^\Mwf$ is quite complex (it is $\boldsymbol{\Pi}^1_1$).
\end{itemize}
\end{example}

We say that $\rho\subseteq[\omega]^{\aleph_0}$ is an $\omega$-\textit{splitting family} if $\forall A\subseteq[\omega]^{\aleph_0}\exists x\in\rho\forall y\in A(x\not\propto y)$. Define $\sfrak_\omega:=\min\{|\rho|:\rho\textrm{\ is $\omega$-splitting}\}$.

\section{The concepts of  \texorpdfstring{$\dfrak(\Rbf)_{\Nbf}^{\Mbf}$}{} and \texorpdfstring{$\bfrak(\Rbf)_{\Nbf}^{\Mbf}$}{}}\label{bd}

The next result show that all standard proof of inequalities between cardinals in Cicho\'n's diagram can be adapted to \autoref{modeltukey}. 

\begin{theorem}[{\cite{b},\cite{fr},\cite{bar} and \cite{paw85}}]\label{cicr}
Let $\Nbf, \Mbf$ be models of ZFC. Then  
\begin{itemize}
    \item[(a)] $\Cbf_\Mwf\preceq_{\Nbf,\Nbf}^{\Mbf,\Mbf}\Cbf_\Nwf^\perp$.  
    \item[(b)] $\Cbf_\Mwf\preceq_{\Nbf,\Nbf}^{\Mbf,\Mbf}\Dbf$. 
    \item[(c)] $\Dbf^{\perp}\preceq_{\Nbf,\Nbf}^{\Mbf,\Mbf}\Dbf\preceq_{\Nbf,\Nbf}^{\Mbf,\Mbf}\Mwf$. 
    \item[(d)] $\Cbf_\Mwf^\perp\preceq_{\Nbf,\Nbf}^{\Mbf,\Mbf}\Mwf$ and $\Cbf^{\perp}_\Nwf\preceq_{\Nbf,\Nbf}^{\Mbf,\Mbf}\Nwf$.
    \item[(e)] $\Mwf\preceq_{\Nbf,\Nbf}^{\Mbf,\Mbf}\Nwf$.
    \item[(f)] $\Mwf\preceq_{\Nbf,\Nbf}^{\Mbf,\Mbf}(\Cbf^{\perp}_\Mwf;\Dbf)$
\end{itemize}
As consequence we get,

\begin{itemize}
    \item[(i)] $\non(\Nwf)_\Nbf^\Mbf\Rightarrow\cov(\Mwf)_\Nbf^\Mbf$ and $\non(\Mwf)_\Nbf^\Mbf\Rightarrow\cov(\Nwf)_\Nbf^\Mbf$.
    \item[(ii)] $\dfrak_\Nbf^\Mbf\Rightarrow\cov(\Mwf)_\Nbf^\Mbf$ and $\non(\Mwf)_\Nbf^\Mbf\Rightarrow\bfrak_\Nbf^\Mbf$.
    \item[(iii)] $\bfrak_\Nbf^\Mbf\Rightarrow\add(\Mwf)_\Nbf^\Mbf$ and $\cof(\Mwf)_\Nbf^\Mbf\Rightarrow\dfrak_\Nbf^\Mbf$ and $\dfrak_\Nbf^\Mbf\Rightarrow\bfrak_\Nbf^\Mbf$.
    \item[(iv)] $\cof(\Mwf)_\Nbf^\Mbf\Rightarrow\non(\Mwf)_\Nbf^\Mbf$ and $\cov(\Mwf)_\Nbf^\Mbf\Rightarrow\add(\Mwf)_\Nbf^\Mbf$ and $\cof(\Nwf)_\Nbf^\Mbf\Rightarrow\non(\Nwf)_\Nbf^\Mbf$ and $\cov(\Nwf)_\Nbf^\Mbf\Rightarrow\add(\Nwf)_\Nbf^\Mbf$.
    \item[(v)] $\cof(\Nwf)_\Nbf^\Mbf\Rightarrow\cof(\Mwf)_\Nbf^\Mbf$ and $\add(\Mwf)_\Nbf^\Mbf\Rightarrow\add(\Nwf)_\Nbf^\Mbf$.
    \item[(vi)] $\non(\Mwf)_\Nbf^\Mbf$ and $\dfrak_\Nbf^\Mbf$ $\Rightarrow\cof(\Mwf)_\Nbf^\Mbf$.
    \end{itemize}
    As a consequence of (iii), (iv) and (vi),
    \begin{itemize}
       \item[(vii)] $\cof(\Mwf)_\Nbf^\Mbf\Leftrightarrow\non(\Mwf)_\Nbf^\Mbf$ 
    and $\dfrak_\Nbf^\Mbf$.
    \end{itemize}
\end{theorem}
Though $\add(\Mwf)_\Nbf^\Mbf$ does not transform into $\bfrak_\Nbf^\Mbf$ and $\cov(\Mwf)_\Nbf^\Mbf$,  the following lemma gives the required characterization.

\begin{lemma}[{\cite[Cor. 1.3]{paw86}}]\label{cor:charadd}
$\neg\add(\Mwf)_\Nbf^\Mbf$ iff there exists $c\in\Mbf$, a Cohen real over $\Nbf$, and a function in $\omega^\omega\cap\Mbf$ which dominates any function from $\Nbf[c]$.
\end{lemma}

\autoref{modelCichon} shows the implications between the sentences of the cardinal characteristics associated with $\Mwf$, $\Nwf$, $\bfrak$ and $\dfrak$.
\begin{figure}[!h]
 \begin{center}
    \includegraphics[scale=0.935]{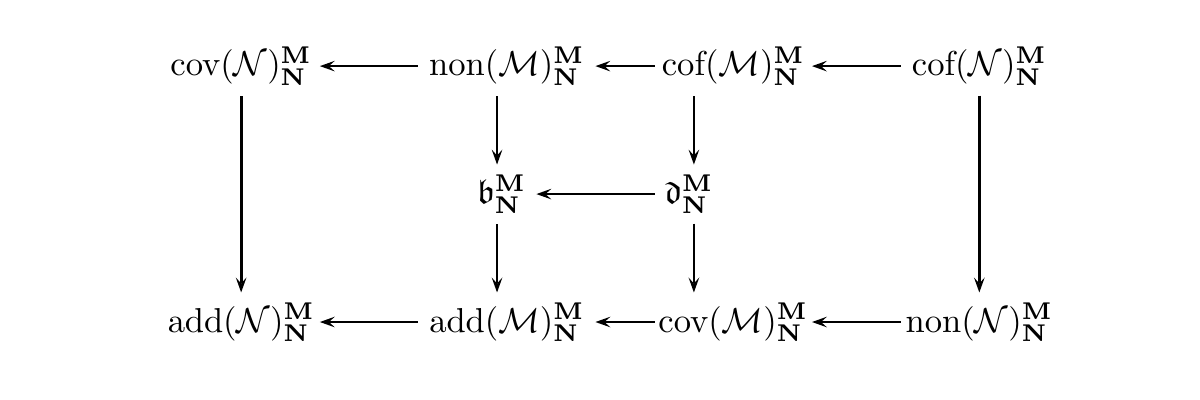}
     \caption{Cicho\'n's diagram with models $\Nbf\subseteq\Mbf$ of ZFC.}
     \label{modelCichon}
\end{center}
\end{figure}

From now on, fix a Suslin ccc poset $\Por$. One interesting case of \autoref{modeltukey} is when  $\Rbf\preceq_{\Nbf,\Nbf[\Por\cap G]}^{\Mbf,\Mbf[G]}\Rbf'$ where $\Nbf\subseteq\Mbf$ and $G$ is a $\Por$-generic over $\Mbf$. A similar definition holds for $\Mbf^\Por$ and $\Nbf^\Por$, more concretely:



\begin{definition}\label{def:tukername}
Fix two models $\Nbf\subseteq\Mbf$ of ZFC. Let $\Por$ be a Suslin ccc forcing notion, let $\Rbf$ and $\Rbf'$ be two real definable relational systems. 

\begin{itemize}
    \item[(1)] $\Rbf\preceq_{\Nbf,\Nbf^{\Por}}^{\Mbf,\Mbf^{\Por}}\Rbf'$ if there are maps $\Psi_-:X^\Mbf\to \Mbf^{\Por}_{X'}$ and $\Psi_+:\Mbf^{\Por}_{Y'}\to Y^\Mbf$ such that 
\begin{itemize}
    \item[(1.1)] for any $x\in X^\Mbf$ and for any $y'\in \Mbf^{\Por}_{Y'}$,  $\Vdash``\Psi_-(x) \sqsubset y'$ implies $x \sqsubset' \Psi_+(y')"$. 
    \item[(1.2)] $\Psi_-[X^\Nbf]\subset \Nbf^{\Por}_{X'}$ and $\Psi_+[\Nbf^{\Por}_{Y'}]\subset Y^\Nbf$.
\end{itemize}
    \item[(2)]  $\Rbf\preceq_{\Nbf^\Por,\Nbf}^{\Mbf^\Por,\Mbf}\Rbf'$ if there are maps $\Psi_-:\Mbf^{\Por}_{X}\to X'^\Mbf$ and $\Psi_+:(Y')^ \Mbf\to \Mbf^{\Por}_{Y}$ such that 
    \begin{itemize}
    \item[(2.1)] for any $x\in \Mbf^{\Por}_{X}$ and for any $y'\in (Y')^\Mbf$,  $\Psi_-(x) \sqsubset y'$ implies $\Vdash x \sqsubset' \Psi_+(y')$. 
    \item[(2.2)] $\Psi_-[\Nbf^{\Por}_{X}]\subset (X')^\Nbf$ and $\Psi_+[(Y')^\Nbf]\subset \Nbf^{\Por}_{Y}$.
\end{itemize}
\end{itemize}
\end{definition}

The next result shows the main reason for considering 
the above definition.   

\begin{lemma}
Fix $\Nbf\subseteq\Mbf$ models of ZFC. Let $G$ be a $\Por$-generic over $\Mbf$. 
\begin{itemize}
    \item[(i)] If  $\Rbf\preceq_{\Nbf,\Nbf^{\Por}}^{\Mbf,\Mbf^{\Por}}\Rbf'$ then $\Rbf\preceq_{\Nbf,\Nbf[\Por\cap G]}^{\Mbf,\Mbf[G]}\Rbf'$.
    \item[(ii)] If  $\Rbf\preceq_{\Nbf^{\Por},\Nbf}^{\Mbf^{\Por},\Mbf}\Rbf'$ then $\Rbf\preceq_{\Nbf[\Por\cap G],\Nbf}^{\Mbf[G],\Mbf}\Rbf'$.
\end{itemize}
\end{lemma}

Now we give an application of \autoref{def:tukername}. For this, consider the following definition:  For a relational system $\Rbf=\la X, Y,\sqsubset\ra$ we let $\Rbf_\omega:=\la X, [Y]^\omega, \sqsubset_\omega\ra$ where $x \sqsubset_\omega \Bar{y}$ if $\exists n<\omega(x \sqsubset y_n)$, which is a relational system.

\begin{definition}[\cite{JS}]\label{def:goodness}
A notion forcing $\Por$ is $\Rbf$-\textit{good} if, for any $\Por$-name $\dot{h}$ for a member of $Y$, there exists  a non-empty countable set $H\subseteq Y$ (in the ground model) such that, for any $x\in X$, if $x$ is $\Rbf$-unbounded over  $H$ then $\Vdash x \not\sqsubset \dot{h}$.
\end{definition}

\begin{lemma}\label{lem:gnltukeyposet}
If $\Por$ is $\Rbf$-good iff  $\Rbf^{\perp}\preceq_{\Vbf^\Por,\Vbf}^{\Vbf^\Por,\Vbf}\Rbf^{\perp}_\omega$ when $\Psi_+$ is the identity map.
\end{lemma}
\begin{proof}
Let $\dot{h}$ be a $\Por$-name for a member of $Y$. Then, in $\Vbf$, choose a non-empty countable set $\bar y_{\dot h}\subseteq Y$ such that, for any $x\in X$, if $x$ is $\Rbf$-unbounded over  $\bar y_{\dot h}$ then $\Vdash x \not\sqsubset \dot{h}$. Put $\Psi_-(\dot h)=\bar y_{\dot h}$ and $\Psi_+(x)=x$ for any $x\in X$. We check \autoref{def:tukername}(2.1). Let $h\in \Vbf^{\Por}_{Y}$ and let  $x\in X^\Vbf$. Assume $\bar y_{\dot h}\sqsubset^\perp_\omega x$. Then $\neg(x \sqsubset_\omega \bar y_{\dot h})$, that is, $x\not\sqsubset y_n$ for all $n<\omega$. Hence, $\Vdash x \not\sqsubset \dot{h}$. 

For the converse, suppose $\Rbf^{\perp}\preceq_{\Vbf^\Por,\Vbf}^{\Vbf^\Por,\Vbf}\Rbf^{\perp}_\omega$ when $\Psi_+$ is the identity map. We want to see that \autoref{def:goodness} holds. To this end, let $\dot{h}$ be a $\Por$-name for a member of $Y$. According to \autoref{def:tukername} choose  $\Psi_-:\Vbf^{\Por}_{Y}\to [Y]^\omega$. So, put $H:=\ran(\Psi_-(\dot h))$. It is not hard to see that $H$ works.
\qedhere{\textrm{\ (\autoref{lem:gnltukeyposet}})}
\end{proof}

\begin{remark}\label{remark}
The maps $\Psi_-:\Vbf^{\Por}_{Y}\to [Y]^\omega$ and $\Psi_+:X\to \Vbf^{\Por}_{X}$ from \autoref{lem:gnltukeyposet} where $\Psi_+$ is the identity  are definable and absolute.
\end{remark}

The following is a consequence of \autoref{lem:gnltukeyposet} and \autoref{remark}.

\begin{lemma}\label{lem:tukeyposet}
$\Mbf\models\Por$ is $\Rbf$-good  iff  $\Rbf^{\perp}\preceq_{\Mbf^\Por,\Mbf}^{\Mbf^\Por,\Mbf}\Rbf^{\perp}_\omega$ when $\Psi_+$ is the identity map.
\end{lemma}

\autoref{lem:tukeyposet} can be generalized to two models.

\begin{lemma}\label{lem:twogoodeness}
Let $\Por$ be a Suslin ccc notion. Asumme that $\Mbf\models$ ``$\Por$ is $\Rbf$-good" and $\Nbf\models$ ``$\Por$ is $\Rbf$-good". Then $\Rbf^{\perp}\preceq_{\Nbf^\Por,\Nbf}^{\Mbf^\Por,\Mbf}\Rbf^{\perp}_\omega$.
\end{lemma}
\begin{proof}
Since $\Nbf\models$ $\Por$ is $\Rbf$-good, we can find $\Psi_-^0:\Nbf^{\Por^\Nbf}_{Y}\to(Y^\omega)^\Nbf$ and $\Psi_-^1:\Nbf^{\Por^\Mbf}_{Y}\to(Y^\omega)^\Mbf$ by \autoref{lem:tukeyposet}. To define  $\Psi_-:\Mbf^{\Por^\Mbf}_{Y}\to(Y^\omega)^\Mbf$ it suffices to note the following: 
If $\dot y\in\Nbf^{\Por^\Nbf}_{Y}$ then 
\[\Nbf\models\underbrace{\forall x\in X(\Psi_-^0(\dot y)\not\sqsubset_\omega x\Rightarrow\Vdash \dot y\not\sqsubset x)}_{\Pi_1^1}\]
By absolutness,
\[\Mbf\models\underbrace{\forall x\in X(\Psi_-^0(\dot y)\not\sqsubset_\omega x\Rightarrow\Vdash \dot y\not\sqsubset x)}_{\Pi_1^1}\]
So we set $\Psi_-:\Mbf^{\Por^\Mbf}_{Y}\to(Y^\omega)^\Mbf$ by 
\[\Psi_-(\dot y):=\left\{\begin{array}{ll}
  \Psi_-^0(\dot y)   &   \text{if $\dot y\in\Nbf^{\Por^\Nbf}_{Y}$,}\\
\Psi_-^1(\dot y) & \text{otherwise.}
\end{array}\right.\]
\qedhere{\textrm{\ (\autoref{lem:twogoodeness}})}
\end{proof}

\begin{remark}\label{remarktukey}
\begin{itemize}
    \item[(i)] The Polish relational systems $\Dbf$ and $\Lc^*$ fulfill $\Dbf\eqT\Dbf_\omega$ and $\Lc^*\eqT(\Lc^*)_\omega$, respectively.
    \item[(ii)] If $\Rbf$ is a Polish relational system then $\Rbf_\omega$ is one as well. It is not hard to see that $\Rbf_\omega\eqT(\Rbf_\omega)_\omega$.
\end{itemize}
\end{remark}

\autoref{lem:twogoodeness} and \autoref{remarktukey} gives us:

\begin{corollary}\label{cohend}
$\Rbf^{\perp}\preceq_{\Nbf^\Cor,\Nbf}^{\Mbf^\Cor,\Mbf}\Rbf^{\perp}_\omega$ holds 
for any Polish relational system $\Rbf$, in particular for $\Dbf$, and $\Lc^*$. As a consequence,
\begin{itemize}
    \item[(i)] $\dfrak_{\Nbf[c]}^{\Mbf[c]}\Rightarrow\dfrak_\Nbf^\Mbf$ and $\bfrak_\Nbf^\Mbf\Rightarrow\bfrak_{\Nbf[c]}^{\Mbf[c]}$.
    \item[(ii)]  $\cof(\Nwf)_{\Nbf[c]}^{\Mbf[c]}\Rightarrow\cof(\Nwf)_\Nbf^\Mbf$ and $\add(\Nwf)_\Nbf^\Mbf\Rightarrow\add(\Nwf)_{\Nbf[c]}^{\Mbf[c]}$.
\end{itemize}
Moreover, for $\Rbf\in\{\Cf,\Ed,\Sbf\}$, we get 
\begin{itemize}
 \item[(iii)]  $\non(\Mwf)_{\Nbf}^{\Mbf}\Rightarrow\non(\Mwf)_{\Nbf[c]}^{\Mbf[c]}$ and $\cov(\Mwf)_{\Nbf[c]}^{\Mbf[c]}\Rightarrow\cov(\Mwf)_{\Nbf}^{\Mbf}$.
    \item[(iv)]  $\rfrak_{\Nbf[c]}^{\Mbf[c]}\Rightarrow\rfrak_{\Nbf}^{\Mbf}$ and $(\sfrak_{\omega})_{\Nbf}^{\Mbf}\Rightarrow(\sfrak_{\omega})_{\Nbf[c]}^{\Mbf[c]}$.
     
     \item[(v)] $\cof(\Mwf)_{\Nbf[c]}^{\Mbf[c]}\Rightarrow\cof(\Mwf)_{\Nbf}^{\Mbf}$ and $\add(\Mwf)_{\Nbf}^{\Mbf}\Rightarrow\add(\Mwf)_{\Nbf[c]}^{\Mbf[c]}$.   
\end{itemize}
\end{corollary}
\begin{proof}
It follows because $\Cor$ is $\Rbf$-good for $\Rbf\in\{\Dbf,\Lc^*,\Sbf,\Ed,\Cf\}$. For more details to see e.g. \cite{BJ}.
\qedhere{\textrm{\ (\autoref{cohend}})}
\end{proof}


The next result concerns random forcing.

\begin{lemma}\label{radomcov} 
\begin{itemize}
    \item[(a)] $(\Lc^*)^\perp\preceq_{\Nbf^{\Bor},\Nbf}^{\Mbf^{\Bor},\Mbf}(\Lc^*)^\perp$.
    \item[(b)] $\Dbf^\perp\preceq^{\Mbf^{\Bor},\Mbf}_{\Nbf^{\Bor},\Nbf}\Dbf^\perp$.
\end{itemize}

As a consequence we get 
\begin{itemize}
    \item[(i)] $\add(\Nwf)_{\Nbf}^{\Mbf}\Rightarrow\add(\Nwf)_{\Nbf[r]}^{\Mbf[r]}$ and $\cof(\Nwf)_{\Nbf[r]}^{\Mbf[r]}\Rightarrow\cof(\Nwf)_{\Nbf}^{\Mbf}$.
    \item[(ii)] $\dfrak_{\Nbf[r]}^{\Mbf[r]}\Rightarrow\dfrak_\Nbf^\Mbf$ and $\bfrak_\Nbf^\Mbf\Rightarrow\bfrak_{\Nbf[r]}^{\Mbf[r]}$.
\end{itemize}
\end{lemma}
\begin{proof}
It follows because $\Bor$ is $\Lc^*$-good and $\Dbf$-good. For more details to see e.g. \cite{BJ}.
\qedhere{\textrm{\ (\autoref{radomcov}})}
\end{proof}

\section{The effect on \texorpdfstring{$\dfrak(\Rbf)_{\Nbf}^{\Mbf}$}{} and \texorpdfstring{$\bfrak(\Rbf)_{\Nbf}^{\Mbf}$}{} after of adding  one Cohen real (random real) without goodness}\label{dbcohenrandom}

For this section assume that $\Nbf\subseteq \Mbf$ are models of ZFC. From now on assume that $c$ and $r$ are a Cohen real over $\Mbf$ and a random real over $\Mbf$, respectively. 

For an increasing function $f\in\omega^\omega$ and a function $x\in2^\omega$ define $x_f\in2^\omega$ as $x_f(n):=x(f(n))$ for $n\in\omega$.

\begin{lemma}\label{cohencd}
$\Dbf\preceq_{\Nbf,\Nbf^\Cor}^{\Mbf,\Mbf^\Cor}\Cbf_\Mwf$. In particular, $\cov(\Mwf)_{\Nbf[c]}^{\Mbf[c]}\Rightarrow\dfrak_\Nbf^\Mbf$ and $\bfrak_\Nbf^\Mbf\Rightarrow\non(\Mwf)_{\Nbf[c]}^{\Mbf[c]}$.
\end{lemma}
\begin{proof}
Let's assume that $\Cor=2^{<\omega}. $Let $\dot{A}$ be a $\Cor$-name for a meager set in $2^\omega$. Find a sequence of $\Cor$-names $\la\dot{A}_n\ra_{n<\omega}$ such that $\Vdash \dot{A}=\bigcup_{n<\omega}\dot{A}_n$ and $\Vdash \dot{A}_n$ is nowhere dense for each $n\in\omega$, and $\la \dot A_n\ra_{n<\omega}$ is increasing.
Since $\Cor$ is countable, $\Cor=\{p_m:m<\omega\}$. For $m,n<\omega$ we can find $q_{m,n}\leq p_m$ and $\sigma_{m,n}\in2^{<\omega}$ such that  $\forall \vartheta\in2^{n}(q_{m,n}\Vdash[\vartheta{}^{\smallfrown}\sigma_{m,n}]\cap\dot{A}_n=\emptyset)$.
Define $g_m\in\omega^\omega$ by $g_m(n):=|q_{m,n}|$. Let $\Psi_+(\dot{A})$ be a function in $\omega^\omega$ which dominates all $g_m$ and let $\Psi_-(f):=\dot c_f$ for $f\in\omega^\omega$. 

It remains to check that, $f\not\leq^* \Psi_+(\dot{A})$ implies $\Vdash c_f\notin \dot{A}$. To see this assume $f\not\leq^* \Psi_+(\dot{A})$. To guarante that  $\Vdash c_f\notin \dot{A}$ it sufficies to prove that, given $p\in\Cor$ and $i<\omega$ there is some $q\leq p$ such that $q\Vdash \dot{c}_f\notin \dot{A}_i$. Let $p\in\Cor$ and $i<\omega$. Choose $m$ such that $p=p_m$ and choose $n\geq i$ such that $f(n)>\Psi_+(A)(n)\geq g_m(n)$.  Wlog assume that $f$ is stricly increasing, so $f(k)\geq f(n)>|q_{m,n}|$ for any $k\geq n$. Find $q\leq q_{m,n}$ such that $q\Vdash \dot{c}_f{\upharpoonright} [n,n+|\sigma_{m,n}|)=\sigma_{m,n}$. Then $q\Vdash \dot{c}_f\notin\dot{A}_n$, so $q\Vdash \dot{c}_f\notin\dot{A}_i$ (because $\la \dot A_n\ra_{n<\omega}$ is increasing and $n\geq i$).
\qedhere{\,\textrm{(\autoref{cohencd}})}
\end{proof}

\begin{lemma}\label{cohencovc}
$\Cbf_\Nwf\preceq_{\Nbf^\Cor,\Nbf}^{\Mbf^\Cor,\Mbf}\Cbf_\Cwf^\perp$. In particular, $\non(\Cwf)_{\Nbf}^{\Mbf}\Rightarrow\cov(\Nwf)_{\Nbf[c]}^{\Mbf[c]}$ and $\non(\Nwf)_{\Nbf[c]}^{\Mbf[c]}\Rightarrow\cov(\Cwf)_{\Nbf}^{\Mbf}$.
\end{lemma}
\begin{proof}
Work in $\Nbf$: Let $B\subseteq2^{\omega}\times2^{\omega}\times2^{\omega}$ be a Borel set such that any Borel function $f:2^\omega\to2^\omega$ fulfills 
\begin{itemize}
\item[(i)] $\{\la x,y\ra:\la x,y,f(y)\ra\notin B\}\in[2^\omega]^{\leq\aleph_0}\times\Mwf$ (Fubini product of ideals),
\item[(ii)] for any $x,y\in2^\omega$, $B_{\la x,y\ra}\in\Nwf$, and  
\item[(iii)] $\{x:A^f_x\notin\Mwf\}$ is a countable set where $A^f:=\{\la x,y\ra:f(y)\not\in B_{\la x,y\ra}\}$. 
\end{itemize}
Such a $B$ exists by \cite[Thm. 1.1]{cicpal}. Let $\dot t\in\Mbf^\Cor_{2^\omega}$. Choose a Borel function $f$ coded in $\Mbf$ such that $\Vdash\dot t=f(\dot c)$. Finally, define $\Psi_-(\dot t):=\{x:A^f_x\notin\Mwf\}$ and $\Psi_+(x):=B_{\la x,\dot c\ra}$ for $x\in2^\omega$.

To finish the proof it remain to check that $\Psi_-(\dot t)\not\ni x$ implies $\Vdash\dot t\in\Psi_+(x)$. Assume that $\Psi_-(\dot t)\not\ni x$, that is, $A^f_x$ is a meager set. Since $c$ is a Cohen real over $\Mbf$, $\Vdash\dot c\notin A^f_x$. Then $\Vdash\dot t=f(\dot c)\in B_{\la x,\dot c\ra}$.
\qedhere{\textrm{\ (\autoref{cohencovc}})}
\end{proof}

The following lemma shows the behaviour of the additivities and cofinalities after adding a single Cohen real.

\begin{lemma}\label{addcohen} 
 $\Mwf\preceq_{\Nbf^\Cor,\Nbf}^{\Mbf^\Cor,\Mbf}\Mwf$. 
As a consequence, 
 $\cof(\Mwf)_{\Nbf}^{\Mbf}\Rightarrow\cof(\Mwf)_{\Nbf[c]}^{\Mbf[c]}$ and $\add(\Mwf)_{\Nbf}^{\Mbf}\Rightarrow\add(\Mwf)_{\Nbf[c]}^{\Mbf[c]}$.
\end{lemma}
\begin{proof}
Let $\dot C\in\Mbf^\Cor_{\Mwf}$. Choose a $C'\subseteq2^\omega\times2^\omega$ meager coded in $\Mbf$ such that $\Vdash C'_{\dot c}=\dot C$. By Sirkoski's isomosphism theorem \cite[Thm. 32.5]{sik}), there is a Borel isomosphism $\varphi:2^\omega\to2^\omega\times2^\omega$  such that $A\in\Mwf$ iff $\varphi(A)\in\Mwf\times\Mwf$ (Fubini product of ideals). Next define $\Psi_-(\dot C):=\varphi^{-1}[C']$.

Let $E$ be a Borel set in $\Mwf\cap\Mbf$. Then $\varphi(E)$ is meager in $2^\omega\times2^\omega$, so put $\Psi_+(E):=\varphi(E)_{\dot c}$.
 
It is clear that $\Psi_-(\dot C)\subseteq E$ implies $\Vdash\dot C\subseteq\Psi_+(E)$.
\qedhere{\textrm{\ (\autoref{addcohen}})}
\end{proof}

The next lemma is the converse of \autoref{radomcov}, that is, it describes $\bfrak$ and $\dfrak$ in the extension obtained by adding a single random real.

\begin{lemma}\label{DBradom}
$\Dbf\preceq_{\Nbf^{\Bor},\Nbf}^{\Mbf^{\Bor},\Mbf}\Dbf$.
In particular,
$\dfrak_\Nbf^\Mbf\Rightarrow\dfrak_{\Nbf[r]}^{\Mbf[r]}$ and $\bfrak_{\Nbf[r]}^{\Mbf[r]}\Rightarrow\bfrak_\Nbf^\Mbf$.
\end{lemma}
\begin{proof}
Let $\dot g\in\Mbf^\Bor_{\omega^\omega}$. By $\omega^\omega$-bounding find $h_{\dot g}\in\Mbf$ such that $\Vdash \dot g\leq^* h_{\dot g}$. Next we define $\Psi_-(\dot g):=h_{\dot g}$ and $\Psi_+(f):=f$ for $f\in\omega^\omega\cap\Mbf$. It is clear that $\Psi_-(\dot g)\leq^*f$ implies $\Vdash \dot g\leq^*\Psi_+(f)$.
\qedhere{\,\textrm{(\autoref{DBradom}})}
\end{proof}

In a similar way to \autoref{addcohen} it can be proved an analogous result for ramdon forcing.

\begin{lemma}\label{lem:cofinalities}
$\Nwf\preceq_{\Nbf^\Bor,\Nbf}^{\Mbf^\Bor,\Mbf}\Nwf$. As a consequence ,
 $\cof(\Nwf)_{\Nbf}^{\Mbf}\Rightarrow\cof(\Nwf)_{\Nbf[c]}^{\Mbf[c]}$ and $\add(\Nwf)_{\Nbf}^{\Mbf}\Rightarrow\add(\Nwf)_{\Nbf[c]}^{\Mbf[c]}$.
\end{lemma}






Now we prove the relationship between $\Dbf$ and $\Cbf_\Nwf$ after adding a single random real.
 
\begin{lemma}\label{randomdn}
$\Cbf_\Nwf^{\perp}\preceq_{\Nbf^\Bor,\Nbf}^{\Mbf^\Bor,\Mbf}\Dbf$. In particular, $\dfrak_{\Nbf}^{\Mbf}\Rightarrow\non(\Nwf)_{\Nbf[r]}^{\Mbf[r]}$ and $\cov(\Nwf)_{\Nbf[r]}^{\Mbf[r]}\Rightarrow\bfrak_{\Nbf}^{\Mbf}$. 
\end{lemma}
\begin{proof}
Let $\dot B\in\Mbf^\Bor_{\Nwf}$. Find a Borel null  $A\subseteq 2^\omega\times2^\omega$ such that $\Vdash \dot B=A_{\dot r}$. Since $A$ is a null set, choose sequences $s_n, t_n\in2^{<\omega}$ such that $|s_n|=|t_n|$ and
\[A\subseteq \bigcap_{m<\omega}\bigcup_{n\geq m}[s_n]\times[t_n]\textrm{\ and\ }\sum_{n=1}^{\infty}2^{-|s_n|-|t_n|}<\infty\]
Find an increasing function $\Psi_-(\dot B)\in\omega^\omega$ by indiction on $n$ such that 
\begin{itemize}
    \item[(a)] $j\leq \Psi_-(\dot B)(n)\to |s_j|<\Psi_-(\dot B)(n+1)$.
    \item[(b)] $\sum_{j\geq \Psi_-(\dot B)(n)}\Lb([s_j]\times[t_j])\leq\frac{\Lb([s_n]\times[t_n])}{2^{n+2}} $ (where $\Lb$ denote the Lebesgue measure).
\end{itemize} 
From (a) and (b) it follows that 
\begin{equation*}
\begin{split}
(\star)\,\,\sum\limits_{\Psi_-(\dot B)(n)\leq j< \Psi_-(\dot B)(n+1)}\frac{2^{|\Psi_-(\dot B)^{-1}[|s_j|]|}}{2^{2|s_j|}}&\leq\sum\limits_{\Psi_-(\dot B)(n)\leq j< \Psi_-(\dot B)(n+1)}\frac{2^{n+2}}{2^{2|s_j|}}\\
&\leq\Lb([s_n]\times[t_n]).
\end{split}
\end{equation*}

On the other hand, define $\Psi_+(f):=\dot r_f$ for $f\in\omega^\omega$. 

To conclude the proof it suficies to prove that, $\Psi_-(\dot B)\leq^* f$ implies $\Vdash B\not\ni \dot r_f$. To do this, assume that $\Psi_-(\dot B)\leq^* f$.
 
Work in $\Mbf$. Define \[H:=\Big\{x:\la x,x_{f}\ra\in\bigcap_{m<\omega}\bigcup_{n\geq m}[s_n]\times[t_n]\Big\}\] It sufficies to prove that $H$ has measure zero. Note that 
\[H= \bigcap_{m<\omega}\bigcup_{n\geq m}\Big\{x:\la x,x_{f}\ra\in[s_n]\times[t_n]\ra\Big\}\]
\begin{claim}\label{covb}
For any increasing function $f\in\omega^\omega$ and $s,t \in 2^{<\omega}$, 
\[\Lb\Big(\Big\{x:\la x,x_f\ra\in[s]\times[t]\Big\}\Big)\leq\frac{2^{|f^{-1}(|s|)|}}{2^{|s|+|t|}}\]
\end{claim}
\begin{proof}
For a proof to see \cite[Claim 2.1]{cardona}.
\qedhere{\textrm{\ (\autoref{covb}})}
\end{proof}
We continue the proof of \autoref{randomdn}. By \autoref{covb} and $(\ast)$ it follows that $H$ has measure zero. Since $\Vdash\dot r$ is a random real over $\Mbf$, $\Vdash\dot r\not\in H$ which means that $\Vdash \la \dot r,\dot r_f\ra\not\in A$, that is $\Vdash \dot r_{f}\not\in A_{\dot r}=\dot B$.
\qedhere{\textrm{\ (\autoref{randomdn}})}
\end{proof}

\begin{lemma}[{\cite[Lemma 3.5]{paw86}}]\label{borelrandom}
There  is some $\Bor$-name of a Borel function $\dot F:2^\omega\to 2^\omega$ such that $\Vdash$ for any $A\in\Mwf$, $F^{-1}[A]$ is both meager and coded in $\Vbf$.
\end{lemma}

The following lemma discusses the behavior of the structure $\Cbf_\Mwf$ in the extension by adding one random real. 

\begin{corollary}\label{randomcmp}
$\Cbf_\Mwf\preceq_{\Nbf,\Nbf[r]}^{\Mbf,\Mbf[r]}\Cbf_\Mwf$. In particular, $\cov(\Mwf)_{\Nbf[r]}^{\Mbf[r]}\Rightarrow\cov(\Mwf)_{\Nbf}^{\Mbf}$ and $\non(\Mwf)_{\Nbf}^{\Mbf}\Rightarrow\non(\Mwf)_{\Nbf[r]}^{\Mbf[r]}$.
\end{corollary}
\begin{proof}
By \autoref{borelrandom}, there is a $\Bor$-name of a Borel function $\dot F:2^\omega\to 2^\omega$ coded in $\Nbf[r]$ such that for any $A\in\Mwf\cap \Mbf[r]$, $F^{-1}[A]$ is both meager and coded in $\Mbf$. Work in $\Mbf[r]$.  Let $A$ be a Borel set in $\Mwf\cap\Mbf[r]$. Then $F^{-1}[A]$ is a meager set in $\Mbf$, so define $\Psi_+(A):=F^{-1}[A]$, and $\Psi_-(z):=F(z)$ for $z\in2^\omega$. It is clear that if $z\not\in F^{-1}[A]$ then $F(z)\not\in A$.
\qedhere{\textrm{\ (\autoref{randomcmp}})}
\end{proof}

As a consequence of \autoref{cicr}, \autoref{cor:charadd}, \autoref{DBradom}, and \autoref{randomcmp}, we get: 
\begin{corollary}
$\add(\Mwf)_{\Nbf[r]}^{\Mbf[r]}\Rightarrow\add(\Mwf)_{\Nbf}^{\Mbf}$ and $\cof(\Mwf)_{\Nbf}^{\Mbf}\Rightarrow\cof(\Mwf)_{\Mbf[r]}^{\Mbf[r]}$.
\end{corollary}

Notice that the next result is the converse of \autoref{randomcmp}.

\begin{lemma}\label{CNradom} Let $h\in\omega^\omega$ suc that $\sum_{i<\omega}\frac{1}{h(i)}<\infty$. Then $\Ed\preceq_{\Nbf^\Bor,\Nbf}^{\Mbf^\Bor,\Mbf}\aLc(\omega,h)$. As a consequence, we get $\cov(\Mwf)_\Nbf^\Mbf\Rightarrow\cov(\Mwf)_{\Nbf[r]}^{\Mbf[r]}$ and $\non(\Mwf)_{\Nbf[r]}^{\Mbf[r]}\Rightarrow\non(\Mwf)_{\Nbf}^{\Mbf}$.
\end{lemma}
\begin{proof}  
Let $\dot{g}$ be a $\Bor$-name for a function in $\omega^\omega$. For each $m\in\omega$, let $\la p_n^m:n\in\omega\ra$ be a maximal antichain deciding the value of $\dot g(m)$. Next, let $\varphi_{\dot{g}}$ be the slalom defined by:
\[\varphi_{\dot{g}}(m):=\:\bigg\{k\in \omega:\Lb\Big(\bigcup \Big\{[p_n^m]:n<\omega,\,p_n^m\Vdash \dot g(m)=k\Big\}\Big)>\frac{1}{h(m)}\bigg\}.\]
From the definition of $\varphi_{\dot{g}}$, it is clear that $|\varphi_{\dot{g}}(m)|<h(m)$. Finally, put $\Psi(\dot{g}):=\varphi_{\dot{g}}$ and $\Psi_+(f):=f$ for $f\in\omega^\omega$. 

To complete the proof it suffices to check that $f\not\in^\infty\varphi_{\dot{g}}$ implies $\Vdash \dot{g}\neq^* f$. To this end, let $n\in\omega$ and $p\in\Bor$, such that $f(m)\not\in\varphi_{\dot{g}}(m)$ for all $m\geq n$. Find $k>n$ such that $\sum_{i=k}^\infty\frac{1}{h(i)}<\Lb([p])$. Now, setting $q:=p\smallsetminus\bigcup_{m=k}^{\infty}\bigcup \{[p_n^m]:n<\omega,\,p_n^m\Vdash \dot g(m)=f(m)\Big\}$, we get $q\Vdash \dot{g}(m)\neq \dot{f}(m)$ for all $m\geq k$, which finishes the proof of the lemma. 

\qedhere{\textrm{\ (\autoref{CNradom}})}
\end{proof}

As a consequence \autoref{cicr},  \autoref{cor:charadd} and \autoref{CNradom}, we get:

\begin{corollary}\label{randomcofm}
 $\cof(\Mwf)_{\Nbf[r]}^{\Mbf[r]}\Leftrightarrow\cof(\Mwf)_\Nbf^\Mbf$ and $\add(\Mwf)_\Nbf^\Mbf\Leftrightarrow\add(\Mwf)_{\Nbf[r]}^{\Mbf[r]}$.
\end{corollary}




We conclude this section by proving the relationship betweem $\Dbf$ and $\Sbf$ after adding a single random real.

\begin{lemma}\label{randomds}
$\Dbf_2\preceq_{\Nbf,\Nbf^\Bor}^{\Mbf,\Mbf^\Bor}\Sbf$. In particular, $\rfrak_{\Nbf[r]}^{\Mbf[r]}\Rightarrow\dfrak_\Nbf^\Mbf$ and $\bfrak_\Nbf^\Mbf\Rightarrow\sfrak_{\Nbf[r]}^{\Mbf[r]}$.
\end{lemma}
\begin{proof}
It suffices to find functions $\Psi_-:\Ior\cap\Mbf\to\Mbf^\Bor_{[\omega]^{\aleph_0}}$ and  $\Psi_+:\Mbf^\Bor_{[\omega]^{\aleph_0}}\to\Ior\cap\Mbf$ such that, for any $I\in\Ior\cap\Mbf$ and $B\in\Mbf^\Bor_{[\omega]^{\aleph_0}}$,  $\Vdash``\textrm{if\ }\Psi_-(I)\propto B$ then $I\ntriangleright\Psi_+(B)"$.

Given $I\in\Ior\cap\Mbf$ define $f_I\in\omega^\omega$ by $f_I(n):=\min I_n$ for $n<\omega$, so put 
\[\Psi_-(I):=\bigcup\Big\{[f(n),f(n+1)):\dot r(\min I_n)=1\Big\}.\]

Let $\dot B\in[\omega]^{\aleph_0}$ be a $\Bor$-name. Let $h_{\dot B}$ be the name of the increasing enumeration of $\dot B$. 

Let $J^{\dot B}\in\Ior$ be a $\Bor$-name such that $\Vdash J_n^{\dot B}:=[h_{\dot B}(n),h_{\dot B}(n))$ for $n<\omega$. Choose $J'\in\Ior\cap\Mbf$ such that $\Vdash J^{\dot B}\sqsubseteq^{\Ior} J'$ (such $J'$ exists because $\Bor$ is $\omega^\omega$-bounding). In the end, define \[\Psi_+(\dot B):=J^*\textrm{\ where\ }J_n^{*}:=J_{2n}^{'}\cup J_{2n+1}^{'}.\]

To finish the proof it sufficies to prove that, if $I\triangleright J^*$ then $\Vdash\Psi_-(I)\not\propto \dot B$. To see this, assume that $I\triangleright J^*$, that is, for infinitely many  $n<\omega$, there is some $m$ such that $I_n\supseteq J_m^*$. Next, set $C:=\big\{n:\exists m(I_n\supseteq J_m^*)\big\}$,
which is an infinite set in $\Mbf$. 

In $\Mbf[r]$, since $r$ is random real over $\Mbf$, both sets   $\big\{n\in C:\,r(\min I_n)=0\big\}$ and $\big\{n\in C:\, r(\min I_n)=1\big\}$ are infinite. Consequenly, $\Psi_-(I)$ splits $\dot B$. 
\qedhere{\textrm{\ (\autoref{randomds}})}
\end{proof}

\subsection*{Acknowledgments}
This paper was developed for the conference proceedings corresponding to the first virtual $``$ Set
Theory: Reals and Topology $"$ Workshop that Professor Diego Mej\'ia organized in November 2020. The
author is very thankful to Professor Diego Mej\'ia for letting him participate in such wonderful workshop.

This work was supported by the Austrian Science Fund (FWF) P30666 and the DOC Fellowship of the Austrian Academy of Sciences at the Institute of Discrete Mathematics and Geometry, TU Wien. 

The author is very thankful to his advisor Dr. Diego Mej\'ia for motivating him to work on this topic and for various discussions concerning the results.







{\small
\bibliography{left}
\bibliographystyle{alpha}
}

\end{document}